\documentclass[11pt]{amsart}
\pdfoutput=1

\usepackage[usenames,dvipsnames]{xcolor}
\usepackage{amsmath,amsthm,verbatim,amssymb,amsfonts,amscd, graphicx, enumerate, enumitem,hyperref,xfrac}
\usepackage{graphics}
\usepackage{tikz-cd}
\usepackage{pinlabel}
\usepackage{stackrel}
\usepackage{mathtools}
\usepackage{wrapfig}
\usepackage{caption}

\hypersetup{colorlinks=true,linkcolor=MidnightBlue,citecolor=MidnightBlue, urlcolor=black}

\usepackage[margin=1.375in,headsep=.3in]{geometry}

\newtheorem{theorem}{Theorem}[section]

\newtheorem{lemma}[theorem]{Lemma}
\newtheorem{proposition}[theorem]{Proposition}

\newtheorem*{question*}{Question} 
 
\theoremstyle{definition}

\newtheorem{remark}[theorem]{Remark}

\usepackage[font=small]{caption}
\theoremstyle{definition}

\newcommand{\CP}{\mathbb{C}P}

\let\int\relax
\newcommand{\int}{\mathring}

\DeclareMathOperator{\id}{{id}}

\newcommand{\sxs}{S^2 \! \times \! S^2}
\newcommand{\smallsum}{\raisebox{1pt}{\text{\footnotesize$\#$}}}

\usepackage{fancyhdr}
\usepackage{etoolbox}
\patchcmd{\thebibliography}
  {\settowidth}
  {\setlength{\itemsep}{2pt plus 0.1pt}\settowidth}
  {}{}
\apptocmd{\thebibliography}
  {\small}
  {}{}

\usepackage[affil-it]{authblk} 
\usepackage{lmodern}
\makeatletter
\patchcmd{\@maketitle}{\LARGE \@title}{\fontsize{16}{19.2}\selectfont\@title}{}{}
\makeatother

\author[K.\ Hayden]{Kyle Hayden}
\address{Department of Mathematics and Computer Science, Rutgers University, Newark, NJ 07102, USA}
\email{kyle.hayden@rutgers.edu}
\author[S.\ Kang]{Sungkyung Kang}
\address{Center for Geometry and Physics, Institute for Basic Science (IBS), Pohang 37673, Korea}
\email{sungkyung38@icloud.com}
\author[A.\ Mukherjee]{Anubhav Mukherjee}
\address{Deapartment of Mathematics, Princeton University, Princeton, NJ 08540, USA}
\email{anubhavmaths@princeton.edu}

\title[One is not enough for closed surfaces]{One stabilization is not enough \\ for closed knotted surfaces}

\begin{document}

\ \vspace{-.05in}

\maketitle

\thispagestyle{empty}


\vspace{-.125in}

\begin{center}\small
\textsc{Kyle Hayden, Sungkyung Kang, Anubhav Mukherjee}
\end{center}

\vspace{-.075in}

\begin{abstract}
In this brief note, we show that there exist smooth 4-manifolds (with nonempty boundary) containing pairs of exotically knotted 2-spheres that remain exotic after one (either external or internal) stabilization. It follows that the ``one is enough'' theorem of Auckly-Kim-Melvin-Ruberman-Schwartz does not hold for closed surfaces whose homology classes are characteristic.
\end{abstract}

\bigskip

\section{Introduction}\label{sec:intro}

Exotic phenomena in dimension four is fundamentally unstable. For example,  results of Wall \cite{wall:4-manifolds} and Gompf \cite{gompf:stable} show that exotic compact 4-manifolds become diffeomorphic after \emph{stabilizing} by taking repeated  connected sums with $\sxs$ (in the orientable case) or $S^2 \tilde \times S^2$ (in the non-orientable case). Analogous results hold for pairs of surfaces $\Sigma$ and $\Sigma'$ in a 4-manifold $X$ that are \emph{exotically knotted}, i.e., isotopic through ambient homeomorphisms of $X$ but not ambient diffeomorphisms of $X$. In particular, results of Perron \cite{perron2} and Quinn \cite{quinn:isotopy} imply that such surfaces become smoothly isotopic after sufficiently many \emph{external stabilizations}, where the ambient 4-manifold $X$ itself is stabilized away from $\Sigma$ and $\Sigma'$. More recently, Baykur-Sunukjian  \cite{baykur-sunukjian:stab} showed that exotic surfaces also become smoothly isotopic after repeated \emph{internal stabilization}, i.e., repeatedly summing each surface with a standard $T^2 \subset S^4$.

It remains a fundamental problem to understand the number of stabilizations required to dissolve exotic phenomena, and it is unknown whether or not a single stabilization suffices in the original setting of closed, simply-connected exotic 4-manifolds. One stabilization is enough for most known examples of exotic phenomena (c.f., \cite{auckly:stab,akbulut:knot-surgery,baykur-sunukjian:round,akmr:stable,akbulut:isotopy,kim:stab}), as well as in multiple more general settings \cite{akmrs:one-is-enough,auckly-sadykov,ruberman-strle}. However, a spate of recent results has  revealed contexts in which a single stabilization is not enough  \cite{lin:twist,lin-mukherjee,guth,kang,guth-hayden-kang-park,konno-mukherjee-taniguchi,hayden:atomic}. 

To date, all examples of exotic 4-manifolds or exotic surfaces that remain exotic after one stabilization have had nonempty boundary. 
With an eye towards the closed case, this note considers the intermediate setting of closed surfaces in 4-manifolds with boundary. 

\begin{theorem}\label{main theorem}
There exist smooth 4-manifolds with boundary  containing pairs of exotically knotted 2-spheres that remain exotic after one (external or internal) stabilization.
\end{theorem}


To the authors' knowledge, these are also the first examples of exotic pairs of surfaces that survive both internal and external stabilization. The key input for Theorem~\ref{main theorem} is the second author's recent construction of exotic contractible 4-manifolds that survive  one stabilization with $\sxs$ \cite{kang}. We prove these exotic 4-manifolds become diffeomorphic after a single \emph{twisted} stabilization, and this enlarged 4-manifold contains the desired pair of exotically knotted 2-spheres. Along the way, we use the results of \cite{akmrs:one-is-enough} to see that a single (possibly twisted) stabilization suffices to dissolve  a large class of corks --- namely, those obtained as surgeries along slice disks in $B^4$; see \S\ref{sec:background}.

\begin{remark} \textbf{(a)} 
In \cite{akmrs:one-is-enough}, Auckly--Kim--Melvin--Ruberman--Schwartz use  Gabai's light bulb theorem \cite{gabai:lightbulb} to establish a ``one is enough''-type result: If $X$ is a smooth, simply-connected 4-manifold  and $\Sigma,\Sigma' \subset X$ are smoothly embedded surfaces of the same genus and homology class such that
\begin{itemize}
    \item  [(i)] $\Sigma$ and $\Sigma'$ are each connected surfaces,
    \item [(ii)] $[\Sigma]= [\Sigma']$ is an \emph{ordinary} homology class, i.e., not dual to $w_2(X)$, and
    \item [(iii)] $X\setminus \Sigma$ and $X\setminus \Sigma'$ are simply-connected,
\end{itemize}
then $\Sigma$ and $\Sigma'$ become smoothly isotopic after one external stabilization.  In \cite{lin-mukherjee}, Lin and the third author showed that hypothesis (i) is necessary. In this note, our ambient 4-manifold $X$ is simply-connected and the surfaces are connected and have simply-connected complements, so our examples show that hypothesis (ii) is necessary, i.e., the surfaces must be ordinary. 
\textbf{(b)} In forthcoming work \cite{auckly:internal}, Auckly produces exotic surfaces in \emph{closed} 4-manifolds that remain distinct after internal stabilization.
\end{remark}

\subsection*{Acknowledgements} The authors thank Dave Auckly for helpful comments.  KH is supported by NSF grant DMS-2243128. SK is supported by the Institute for Basic Science (IBS-R003-D1).

\medskip

\section{Some topological constructions}\label{sec:background}

\subsection{Surgery along a slice disk}
We begin by recalling a  construction of contractible 4-manifolds based on slice disks (c.f., \cite{gordon:contractible}). Given a knot $K \subset S^3$ bounding a smooth slice disk $D\subset B^4$, fix a compact tubular neighborhood $N(D)\cong D \times D^2$ and a choice of meridian $\{p\} \times \partial D^2  \subset S^3$ for some point $p \in K =\partial D$. By taking  the disk exterior $B^4 \setminus \mathring{N}(D)$  and attaching a $(-k)$-framed 2-handle along the meridian $\mu$, we obtain a contractible 4-manifold whose boundary is $\smash{S^3_{\sfrac{1}{k}}(K)}$. We refer to this 4-manifold as $1/k$-surgery along the slice disk $D \subset B^4$, and we denote it by $B^4_{\sfrac{{1}}{k}}(D)$.



We can also describe this construction in terms of Kirby diagrams. Given a handle diagram of $B^4$ in which $K$ is unknotted and $D$ is its standard slice disk, we replace $K$ with a dotted 1-handle curve and add a $(-k)$-framed 2-handle along a meridian to $K$. An example is illustrated in Figure~\ref{fig:stevedore}, where $K$ is the stevedore knot.

\begin{figure}[b]
\center
\def\svgwidth{.9\linewidth} 
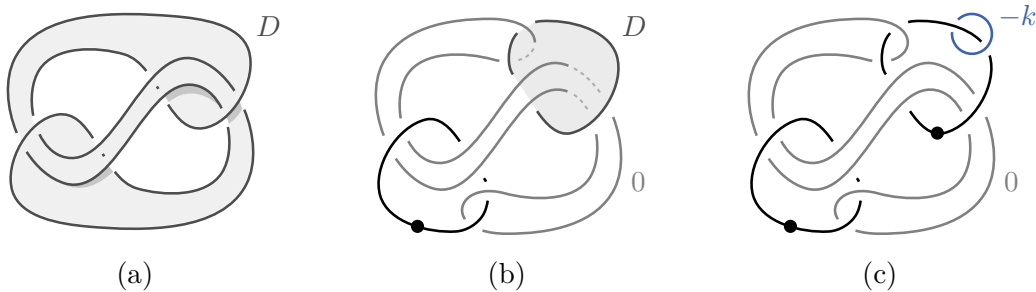
      \captionsetup{width=.84\linewidth}

\medskip

\caption{(a) A slice disk $D$ bounded by the stevedore knot. (b) A  handle diagram for $B^4$ in which $D$ is standard. (c) The 4-manifold $B^4_{\sfrac{{1}}{k}}(D)$.}
\label{fig:stevedore}

\vspace{-.35in}
\end{figure}

\newpage

\subsection{Log transformation} Consider a smooth 4-manifold $X$ containing a smoothly embedded torus $T$ with square zero, and fix a framed tubular neighborhood $N(T) \cong T^2\times D^2$. A \emph{torus surgery} along $T$ consists of removing $\mathring{N}(T)$ and regluing $N(T)$ using a non-trivial  self-diffeomorphism of its boundary $\partial N(T) \cong T^2 \times \partial D^2 = T^3$. Our arguments will involve a special case of this operation known as a \emph{logarithmic transform of multiplicity zero}, where the self-diffeomorphism of $T^3 \cong T^2 \times \partial D^2$ exchanges $\partial D^2$ with an essential simple closed curve in $T^2$. 
Note that the result is independent of the choice of curve in $T^2$, as any two such curves can be exchanged  by a diffeomorphism $f$ of $T^2$ and the gluing diffeomorphism $f \times \id$ of $T^2 \times \partial D^2$ extends to $T^2 \times D^2$.  For a handle-theoretic description of this operation,  see Figure~\ref{fig:log-transform} (c.f., \cite[Figure~8.25]{gompfstipsicz}).  
%

\begin{figure}
\center
\def\svgwidth{.85\linewidth} 
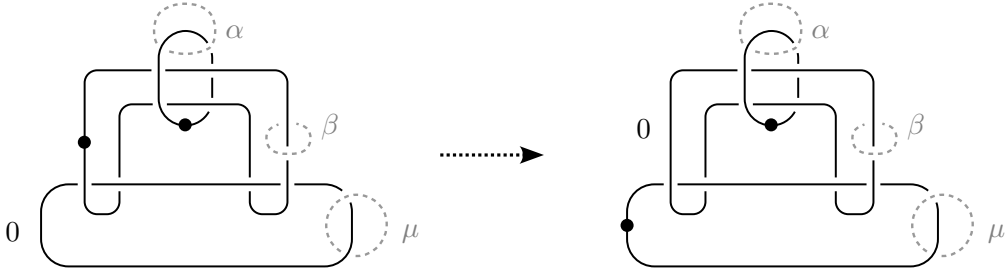
\caption{A log transform of multiplicity zero.}\label{fig:log-transform}

\end{figure}

\medskip

\section{Proof of the main theorem}

We begin by observing that a single (possibly twisted) stabilization suffices for any pair of exotic 4-manifolds obtained as surgeries along a pair of slice disks in $B^4$ with the same boundary.

\begin{lemma}\label{lem:one-is-enough} 
Let $D$ and $D'$ be slice disks in $B^4$ bounded by a knot $K \subset S^3$.

\vspace{1pt}

\begin{enumerate}
    \item [(a)] If $k$ is even, then $B^4_{\sfrac{1}{k}}(D) \smallsum (S^2 \! \times \! S^2) \cong B^4_{\sfrac{1}{k}}(D') \smallsum  (S^2 \! \times \! S^2)$ rel boundary.

\vspace{1pt}

    \item [(b)] If $k$ is odd, then $B^4_{\sfrac{1}{k}}(D) \smallsum (S^2  \tilde \times S^2) \cong B^4_{\sfrac{1}{k}}(D') \smallsum  (S^2  \tilde \times  S^2)$ rel boundary.  
\end{enumerate} 
\end{lemma}

\begin{proof}
Consider the 0-framed knot trace $X_0(K)$ and let  $S$ and $S'$ denote the 2-spheres obtained from $D$ and $D'$, respectively, by taking their union with the core disk in the 2-handle. Let $Z$ denote the 4-manifold obtained from $X_0(K)$ by attaching a $(-k)$-framed 2-handle along the meridian to $K$. Then $B^4_{\sfrac{1}{k}}(D)$ and $B^4_{\sfrac{1}{k}}(D')$ are obtained from $Z$ by surgering along the 2-spheres $S$ and $S'$, respectively. Observe that $S$ and $S'$ have simply-connected complements (because their meridians bound core disks of a 2-handle) and that these represent a characteristic homology class if and only if $k$ is odd.  By \cite{akmrs:one-is-enough}, $S$ and $S'$ are isotopic in $Z$ (rel $\partial Z$) after one stabilization with either $\sxs$ if $k$ is even or $S^2 \tilde \times S^2$ if $k$ is odd. This ambient isotopy (rel boundary) induces a diffeomorphism of the surgered 4-manifolds, namely the once-stabilized copies of $B^4_{\sfrac{1}{k}}(D)$ and $B^4_{\sfrac{1}{k}}(D')$. 
\end{proof}

For the narrower purposes of proving Theorem~\ref{main theorem}, we can get away with the weaker conclusion that $B^4_{\sfrac{1}{k}}(D)$ and $B^4_{\sfrac{1}{k}}(D')$ become diffeomorphic after one (possibly twisted) stabilization where, a priori, the diffeomorphism may be nontrivial along the boundary. (In explicit cases, as in the proof of Theorem~\ref{main theorem}, one may often promote this to the stronger rel boundary conclusion.) 

This weaker alternative admits an elementary, handle-theoretic proof: Begin by fixing a handle structure on $B^4$ relative to the slice disk $D$, so that $D$ appears as an unknotted, boundary-parallel disk in the 0-handle and $K$ appears as an unknot, with additional handles attached to the complement of $D$. This is depicted schematically in part (a) of Figure~\ref{fig:one-is-enough}. Part (b) of the figure depicts $B^4_{\sfrac{1}{k}}(D)$. In part (c), we perform a stabilization, which is twisted if and only if $k$ is odd. After a handle slide, we obtain the diagram in (d). By sliding the $(-k)$-framed 2-handle over the 0-framed meridian in the top right, we obtain the diagram in (e). Sliding all the original 2-handles off of the 1-handle yields the diagram in (f), and a 1-/2-handle cancellation then yields the diagram in (g). Observe that this final diagram is simply the union of $B^4$ with a 0-framed 2-handle attached along $K$ and a $k$-framed 2-handle attached along a meridian $\mu$ of $K$; this 4-manifold only depends  on $K$ and $k$, and not on $D$ or $D'$. 

\begin{figure}
\center
\def\svgwidth{\linewidth} 
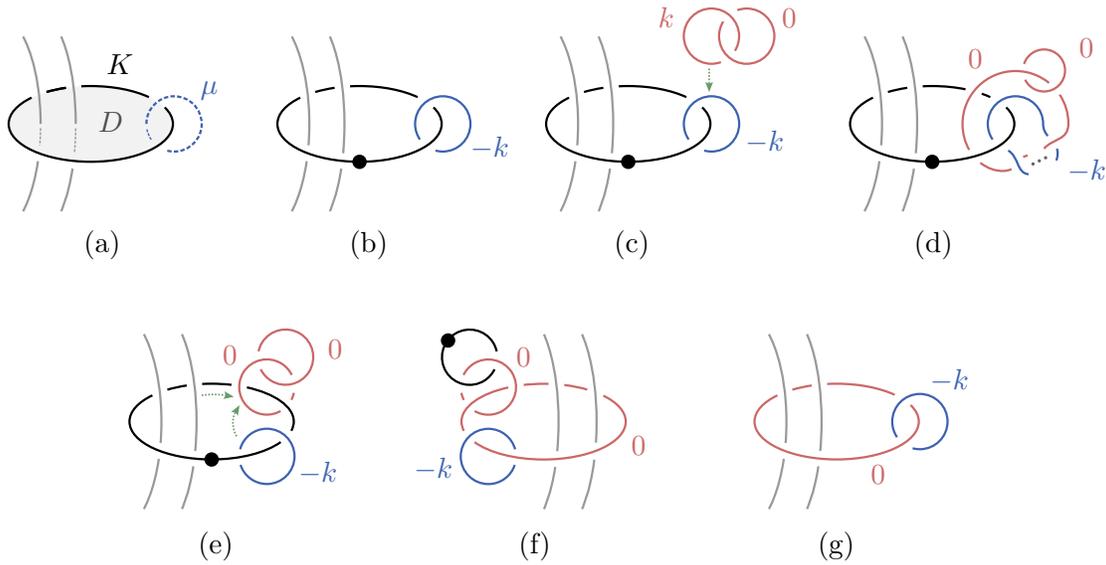

\medskip

\caption{Surgery along a slice disk after one (possibly twisted) stabilization.}\label{fig:one-is-enough}
\end{figure}

\begin{figure}
\center

\vspace{-.05in}

\def\svgwidth{.95\linewidth} 
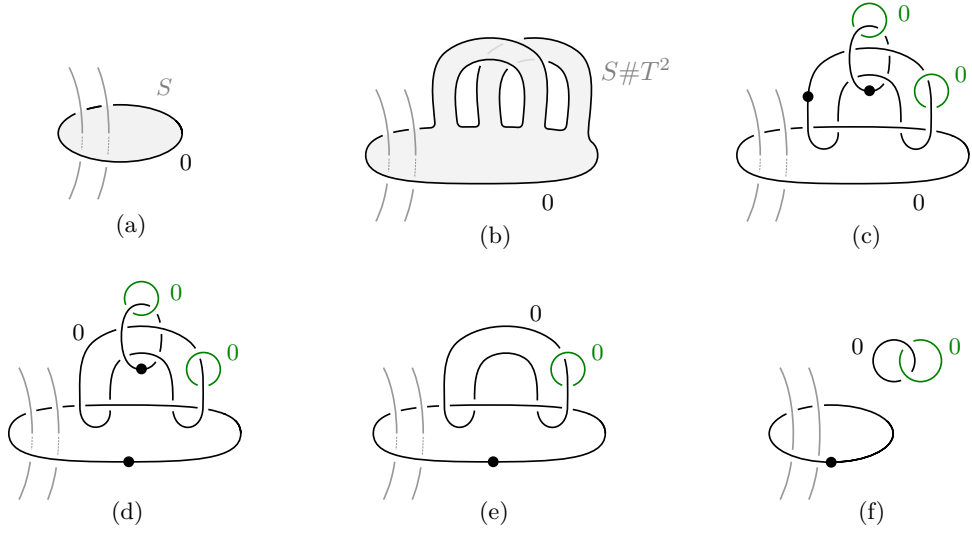

      \captionsetup{width=.875\linewidth}

\caption{Relating internal and external stabilization via logarithmic transform.}\label{fig:torus-surgery}
\end{figure}

\medskip

We also make note of a relationship between internal and external stabilization. The following is well-known to experts, but we sketch its proof for completeness.

\begin{lemma}\label{lem:internal}
Let $S$ be a smoothly embedded 2-sphere of square zero in a smooth 4-manifold $X$, and let $X_0$ denote the result of surgering $X$ along $S$. Then $X_0 \smallsum \sxs$ is obtained from $X$ by a logarithmic transformation of multiplicity zero along $S \smallsum T^2 \subset X$.
\end{lemma}

\begin{proof}
Fix a handle decomposition of $X$ relative to a neighborhood $N(S) \cong S^2 \times D^2$. This is depicted schematically in Figure~\ref{fig:torus-surgery}(a), where the grey arcs depict 2-handles that may intersect $S$. Since the diagram starts with $S^2 \times D^2$ (which is shown in Figure~\ref{fig:torus-surgery}(a) as the 0-trace of an unknotted component) which contains the given sphere $S$ as $S^2 \times \{0\}$, it is clear that while 1-handles and 3-handles can exist, we may assume that they are attached away from $S$.
As illustrated in parts (b) and (c) of the figure, we may locally stabilize $S$ to produce the torus $S \smallsum T^2$ and introduce canceling 1-/2-handle pairs to adapt the diagram to $S \smallsum T^2$. The diagram in (d) depicts the result of a multiplicity-zero log transform. After canceling the topmost 1-/2-handle pair and performing minor isotopy, we obtain (f), which is a diagram for $X_0 \smallsum (\sxs)$. (We also note that it does not matter which $S^1$-factor in $T^2$ was exchanged with the meridian to $S\smallsum T^2$; either transformation yields the same result.)
\end{proof}

\subsection{ Proof of the main theorem } Let $W$ and $W'$ denote the pair of compact, contractible 4-manifolds constructed by the second author in \cite{kang}. These 4-manifolds are homeomorphic but not diffeomorphic, and they remain nondiffeomorphic after a single stabilization   with $\sxs$. The construction begins with  a pair of  $(+1)$-surgeries along disks $D,D' \subset B^4$ with the same boundary $K$. While these 4-manifolds are only exotic rel boundary, results of Akbulut-Ruberman \cite{akbulut2016absolutely} can be used to extend $B^4_{+1}(D)$ and $B^4_{+1}(D')$ by an invertible homology cobordism $C$ from $S^3_{+1}(K)$ to another integer homology sphere $Y$ with the property that every self-diffeomorphism of $Y$ has an extension to $C$ that is the identity on $S^3_{+1}(K)$. The extended 4-manifolds $W=B^4_{+1}(D) \cup C$ and $W'=B^4_{+1}(D') \cup C$ are then \emph{absolutely exotic} and remain so after one stabilization with $\sxs$ \cite{kang}. 

By Lemma~\ref{lem:one-is-enough}, the 4-manifolds $B^4_{+1}(D)$ and $B^4_{+1}(D')$ become diffeomorphic rel boundary after one twisted stabilization with $S^2 \tilde \times S^2$. Since $W'$ can be obtained from  $W$ by cutting out $B^4_{+1}(D)$ and gluing in $B^4_{+1}(D')$ via the identity, we obtain the following:

\smallskip

\begin{proposition} \label{diffeomorphism}
$W \# (S^2 \tilde \times S^2)$ and $W' \#( S^2 \tilde \times S^2)$ are diffeomorphic rel boundary. \hfill $\square$
\end{proposition}


We are now in position to prove the main theorem. 

\begin{proof}[Proof of Theorem~\ref{main theorem}]
Let the knot $K$, the slice disks $D$ and $D'$, and the 4-manifolds $W$ and $W'$ be as described above. We also let $Z$ denote the enlarged 4-manifold $Z=W \# (S^2 \tilde \times S^2) \cong W' \#( S^2 \tilde \times S^2)$. As in the proof of Lemma~\ref{lem:one-is-enough}, $Z$ can also be obtained from $B^4$ by attaching a 0-framed 2-handle along $K$ and a $(-1)$-framed 2-handle along a meridian of $K$. 

Consider the 2-spheres $S$ and $S'$ in $Z$ obtained by capping off the slice disks $D$ and $D'$ with the core of the 0-framed 2-handle attached along $K$. As in the proof of Lemma~\ref{lem:one-is-enough}, these spheres have simply-connected complements, hence are topologically isotopic by results of Perron \cite{perron2} and Quinn \cite{quinn:isotopy} (building on work of Freedman \cite{freedman}).   On the other hand, surgering $Z$ along $S$ or $S'$ yields $W$ or $W'$, respectively, so these spheres are not smoothly isotopic. Moreover, since $W$ and $W'$ remain nondiffeomorphic after one stabilization with $\sxs$, we see that $S$ and $S'$ remain smoothly nonisotopic after one external stabilization. 

If we instead consider internal stabilization, we obtain tori $T=S \smallsum T^2$ and $T'=S' \smallsum T^2$. Towards a contradiction, suppose there is a smooth isotopy carrying $T$ to $T'$. This carries a tubular neighborhood $N(T) \cong T^2 \times D^2$ to a tubular neighborhood $N(T')$. This diffeomorphism preserves the class of the meridian $\partial D^2$ in $\partial N(T)$ and $\partial N(T')$, so the 4-manifolds obtained from $Z$ by log transforms of multiplicity zero along either  $T$ or $T'$ are diffeomorphic. (Note that, as discussed in  \S\ref{sec:background}, the result of the 0-log transform is independent of the choice of framing for these neighborhoods.) However, Lemma~\ref{lem:internal} shows that the resulting 4-manifolds can also be obtained by surgering $Z$ along $S$ or $S'$ and then stabilizing with $\sxs$, yielding  $W \smallsum (\sxs)$ and $W'\smallsum (\sxs)$. This is a contradiction, as these latter 4-manifolds are not diffeomorphic. 
\end{proof}

{\small 
\bibliographystyle{alphamod}  
\bibliography{biblio}
}

\end{document}